\documentclass[12pt]{amsart}

\usepackage{amsfonts,amssymb,amsmath}
\usepackage{graphicx}
\usepackage{psfrag}
\usepackage{color,xcolor}
\usepackage{verbatim}
\usepackage{epstopdf}

\usepackage[colorlinks,citecolor=red,pagebackref,hypertexnames=false]{hyperref}

\newtheorem{thm}{Theorem}
\newtheorem{lem}[thm]{Lemma}

\theoremstyle{remark}

\newcommand{\R}{{\mathbb R}}
\newcommand{\Z}{{\mathbb Z}}

\newcommand{\bfx}{{\mathbf x}}

\newcommand{\card}{{\rm card\,}}
\def\leb{\mathcal{L}}
\newcommand{\dist}{{\rm dist\,}}

\begin{document}

\title{Positive measure and level sets of the Takagi function}

\author{Lai Jiang}
\address{Hangzhou Institute for Advanced Study, UCAS, Hangzhou 310024}
\email{jianglai@ucas.ac.cn}

\subjclass[2020]{Primary 26A27; Secondary 28A80}

\date{}

\keywords{Takagi function, level sets, finite cardinalities, Lebesgue measure.}

\begin{abstract}

Let $T$ be the classical Takagi function, and $L(y):=\{ x \in [0,1] : T(x) = y \}$ be its level set at height $y$.
For each positive integer $m$, let $S_m$ be the set of $y$ for which $L(y)$ has exactly $m$ points.
We prove that $S_{2n}$ has positive Lebesgue measure for every positive integer $n$, confirming a conjecture of Allaart.
\end{abstract}

\maketitle

\section{Introduction}

In 1903, Takagi introduced a continuous nowhere differentiable
function defined by
$$
	T(x) := \sum_{i=0}^{\infty} \frac{\phi(2^i x)}{2^i}, \qquad x \in [0,1].
$$
Here $\phi(x) := \dist ( x, \Z )$ denotes the distance from $x$ to the nearest integer. 
The function $T$ is called the classical Takagi function.

Allaart and Kawamura~\cite{AK11} surveyed the history, properties and applications of the Takagi function.
Among these, the structure of its level sets has attracted particular attention. 
Kahane~\cite{K59} showed that the maximum of $T$ is $2/3$.

For every $y \in \R$, define
$$
	L(y) := \{ x \in [0,1] : T (x)=y \},
$$
which is called the \emph{level set} at height $y$.
Buczolich~\cite{B08} proved that $L(y)$ is finite for Lebesgue almost every $y$ in the range of $T$.
Allaart~\cite{A13} showed that every finite level set has even cardinality by symmetry.
de Amo et al.~\cite{ABDF11} proved that every level set has Hausdorff dimension at most $1/2$.
Lagarias and Maddock introduced local level sets in~\cite{LM12}, and their work was extended by Jiang, Ying and Zhang~\cite{JYZ25}.
For further work on level sets of the Takagi function and its generalizations, see~\cite{A14,LM11,Y20}.

For every $m \in \Z^+$, write
$$
	S_m := \{y\in[0,2/3]:\card L(y)=m\}.
$$
Allaart proved that $S_{2n}$ is nowhere dense~\cite{A12} and uncountable~\cite{A13} for every $n\in\Z^+$.
In~\cite{A13}, he further showed that $S_{2n}$ has positive Lebesgue measure if $2n=2^p\pm2^q$ for some $p,q\in\Z^+$, 
and conjectured that this holds for every $n\in\Z^+$.
We prove this conjecture.


\begin{thm}\label{thm:main}
For every $n\in\Z^+$, we have $\leb (S_{2n}) > 0$.
Here $\leb$ denotes one-dimensional Lebesgue measure.
\end{thm}


\section{Proof of the theorem}

We first recall some notation and results of Allaart~\cite{A13,A11S}.
For every $k\in\Z^+$, we write $t_k := k/2^k$. 
Note that $t_1=t_2= 1/2 $ and that $\{t_k\}_{k\geq 2}$ is strictly decreasing. 
Define $I_2 := [1/2,2/3]$ and
$$
	I_k := [t_k,t_{k-1}), \qquad k \geq 3.
$$
Then $\{ I_k : k \geq 2 \}$ forms a partition of $(0,2/3]$.

\begin{lem}[{\cite[Lemma 5.3]{A11S}}]\label{lem:allaart-sigma}
For every $k \geq 4$, we have
$$
	\leb(I_k\setminus S_2)\leq \frac{1}{2^{k+1}}-\frac{1}{2^{2k-1}}.
$$
\end{lem}

\begin{lem}[{\cite[Theorem 4.5]{A13}}]\label{lem:allaart-S2}
We have $5/12<\leb(S_2)<35/72$.
\end{lem}

For every $y \in I_k$ with $k\geq3$, set
$$
\Psi(y):=4(y-t_k), \qquad \Phi(y):=4^k(y-t_k).
$$
For every $y \in I_2$, we write $\Phi(y) := 4 (y- 1/2 )$.
We recall the following cardinality relations.

\begin{lem}[{\cite[Corollary 3.2]{A13}}]\label{lem:allaart}
Fix $0 < y \leq 2/3$.
If $y \in I_2$, then
$$
	\card L(y) = 2 \sum_{j=0}^{\infty}\card L\big(4^j \Phi(y) \big).
$$
If $y\in I_k$ for some $k \geq 3$, then
$$
	\card L(y) = \card L\big( \Psi(y) \big) + 2 \sum_{j=0}^{\infty}\card L\big( 4^j \Phi(y) \big).
$$
\end{lem}

The following lemma can be directly obtained from~\cite{A13,A11S}.

\begin{lem}\label{lem:reduction}
Let $k\geq 3$ and $y \in I_k$. If $\Phi(y)>1/6$, then
$$
	\card L(y) = \card L\big( \Psi(y) \big) + 2 \card L\big( \Phi(y) \big).
$$
\end{lem}
\begin{proof}
Since $\Phi(y)>1/6$, we have $4^j\Phi(y)>2/3$ for every $j \geq 1$. 
Hence $ L( 4^j \Phi(y) ) =\emptyset$ for all $j \geq 1$.
Then the conclusion follows from Lemma~\ref{lem:allaart}.
\end{proof}


\begin{lem}\label{lem:borel}
For every $m\in\Z^+$, the set $S_m$ is Borel.
\end{lem}

\begin{proof}
Fix $m\in\Z^+$. 
For every $r\in\Z^+$ and $y\in[0,2/3]$, we call $\bfx=(x_1,\ldots,x_m) \in [0,1]^m $ $r$-admissible for $y$ if
$T(x_i)=y$ for every $1 \leq i \leq m$ and $|x_i-x_j| \geq 1/ r $ for all $i \neq j$.
Write
$$
	E_r := \big\{ y\in[0,2/3] : \exists \, \bfx \mbox{ which is } r\mbox{-admissible for } y \big\} .
$$

We first show that $E_{r}$ is closed.
Assume that $\{ y_n \}$ is a sequence contained in $E_r$ such that $ \lim_{n \to \infty} y_n = y $.
For every $n \in \Z^+$, pick $\bfx^{(n)}$ which is $r$-admissible for $y_n$.
Since $[0,1]^m$ is compact, the sequence $\{\bfx^{(n)}\}$ has a subsequence converging to some $\bfx=(x_1,\ldots,x_m)\in[0,1]^m$.

The continuity of $T$ gives $T(x_i)=y$ for every $1 \leq i \leq m$.
Moreover, for every $i\neq j$, $| x_i-x_j |\geq 1/r$.
Thus $\bfx$ is $r$-admissible for $y$, so $y\in E_r$.
It follows that $E_r$ is a Borel set.
Therefore the set
$$
	\big\{ y:\card L(y)\geq m \big\} =\bigcup_{r=1}^{\infty}E_r
$$
is Borel.
Since $m$ was arbitrary,
$$
	S_m =\{y:\card L(y)\geq m\} \setminus\{y:\card L(y)\geq m+1\}
$$
is also Borel.
\end{proof}

\begin{lem}\label{lem:density}
We have
\begin{equation}\label{eq:density-strong}
\lim_{\rho\to0^+}
\frac{\leb\big((0,\rho)\setminus S_2\big)}{\rho}=0.
\end{equation}
\end{lem}

\begin{proof}
For every sufficiently small $\rho>0$, there exists $p\geq4$ such that $t_p\leq\rho<t_{p-1}$.
Since $ \{ I_k : k\geq p \}$ is a partition of $(0,t_{p-1})$, Lemma~\ref{lem:allaart-sigma} gives
$$
0\leq \frac{\leb\big((0,\rho)\setminus S_2\big)}{\rho}
\leq \frac{1}{\rho}\sum_{k=p}^{\infty} \Big(\frac{1}{2^{k+1}}-\frac{1}{2^{2k-1}}\Big)
<\frac{1}{2^p\rho}
=\frac{t_p}{p\rho}
\leq\frac{1}{p}.
$$
When $\rho$ tends to $0^+$, $p$ tends to infinity.
Then the conclusion follows.
\end{proof}

\begin{lem}\label{lem:normalization}
If $m\in\Z^+$ and $\leb(S_m)>0$, then $\leb \big(S_m\cap [1/6,2/3] \big)>0$.
\end{lem}

\begin{proof}
For every $k\geq4$, define a function on $I_k$ by $R_k(y) := t_{k-1}+ y /4$.
Since
$$
t_{k-1}<R_k(y)< 5 t_{k-1} /4 <t_{k-2}, \qquad y \in I_k,
$$
we have $R_k(I_k) \subset I_{k-1}$.
Moreover, we have $\Psi\big(R_k(y)\big)=y$ and
$$
\Phi\big(R_k(y)\big)
 =4^{k-2}y
 \geq k2^{k-4}>\frac{2}{3}.
$$
Therefore Lemma~\ref{lem:reduction} gives
$\card L(R_k(y))=\card L(y)$.
Consequently,
\begin{equation}\label{eq:rec-leb}
\leb(S_m\cap I_{k-1})
\geq \leb\big(R_k(S_m\cap I_k)\big)
=\frac{1}{4}\leb(S_m\cap I_k).
\end{equation}

Recall that $\{ I_k : k \geq 2 \}$ is a partition of $(0,2/3]$. 
Since $\leb(S_m)>0$, there is some $k\geq2$ such that
$\leb(S_m\cap I_k)>0$.
If $k\geq4$, by repeatedly applying \eqref{eq:rec-leb}, we obtain
 $\leb(S_m\cap I_3)>0$.
Thus, in either case,
$$
\leb\big(S_m\cap(I_2\cup I_3)\big)>0.
$$
As $I_2\cup I_3=[3/8,2/3] \subset [1/6,2/3]$, the conclusion follows.
\end{proof}

\begin{lem}\label{lem:binary}
Let $n\in\Z^+$. If $\leb(S_{2n}) > 0$, then we have both $\leb(S_{4n})>0$ and $\leb(S_{4n+2})>0$.
\end{lem}

\begin{proof}
Put $A:=S_{2n}\cap (1 /6 ,2 /3 )$.
By Lemma~\ref{lem:normalization}, we have
$\leb(A)>0$.
Define a function on $(0,2/3)$ by
$$
	D(z):=\frac{1}{2}+\frac{z}{4}.
$$
Then $D(A) \subset I_2$.
For any $z\in A$, $\Phi\big(D(z)\big)=z$. By Lemma~\ref{lem:allaart},
$$
	\card L\big(D(z)\big)=2\sum_{j=0}^{\infty}\card L(4^jz)=2 \card L(z)=4n,
$$
where the second equality follows from $z>1/6$.
Thus $D(A)\subset S_{4n}$ and
$$
	\leb(S_{4n})\geq\leb\big(D(A)\big)=\frac{1}{4}\leb(A)>0.
$$

For every $k\in\Z^+$, write $\rho_k:=4^{1-k} $ and
$$
	A_k:=\{z\in A:\rho_kz\in S_2\}.
$$
Since $A_k \subset A\subset(0,1)$, Lemma~\ref{lem:density} gives
$$
 0\leq \lim_{k \to \infty} \leb(A\setminus A_k)
= \lim_{k \to \infty} \frac{\leb\big(\rho_k A\setminus \rho_k A_k \big)}{\rho_k}
\leq \lim_{k \to \infty} \frac{\leb\big((0,\rho_k)\setminus S_2\big)}{\rho_k}
=0.
$$
Hence $\leb(A_k)>0$ for all sufficiently large $k$.

Fix $k\geq 4$ such that $\leb (A_k) >0$.
Define
$$
	O_k(z):=t_k+\frac{z}{4^k}, \qquad z\in A_k.
$$
Since $z<1$, $t_k <O_k(z) < t_k  +4^{-k} < t_{k-1}$.
Therefore $O_k(z)\in I_k$, with
$$
\Psi\big( O_k(z) \big)=\rho_k z \in S_2
\quad\text{and}\quad
\Phi\big( O_k(z) \big)= z \in S_{2n}.
$$
Since $z>1/6$, Lemma~\ref{lem:reduction} gives
$$
\card L\big(O_k(z)\big) =\card L(\rho_kz)+2\card L(z) =4n+2.
$$
Therefore $O_k(A_k)\subset S_{4n+2}$ and
$$
\leb(S_{4n+2}) \geq \leb\big(O_k(A_k)\big) = 4^{-k} \leb(A_k) >0.
$$
Then the proof is complete.
\end{proof}

\begin{proof}[Proof of Theorem~\ref{thm:main}]
By Lemma~\ref{lem:allaart-S2}, we have $\leb(S_2)>0$.
For every $n\geq2$, write $n=2q$ or $n=2q+1$ with $1\leq q<n$.
We obtain $\leb(S_{2n})>0$ by recursively applying Lemma~\ref{lem:binary}.
\end{proof}

\bibliographystyle{amsplain}

\end{document}